\documentclass{smfart}
\usepackage[utf8]{inputenc}

\usepackage[english]{babel}
\usepackage{graphicx}
\usepackage{amssymb}
\usepackage{amsmath}

\usepackage{mathtools}

\usepackage{verbatim}

\usepackage{leftidx}

\usepackage{amsthm}
\newtheorem{theorem}{Theorem}[section]

\newtheorem{lemma}[theorem]{Lemma}
\newtheorem{definition}[theorem]{Definition}

\newtheorem{question}[theorem]{Question}

\DeclareMathOperator{\C}{\mathbb{C}}

\DeclareMathOperator{\Mod}{Mod}

\DeclareMathOperator{\GL}{GL}
\DeclareMathOperator{\SL}{SL}
\DeclareMathOperator{\Tr}{Tr}

\let\Im\relax
\DeclareMathOperator{\Im}{Im}

\DeclareMathOperator{\lcm}{lcm}

\DeclareMathOperator{\Conj}{Conj}
\DeclareMathOperator{\Sub}{Sub}
\DeclareMathOperator{\D}{D}
\DeclareMathOperator{\CD}{CD}
\DeclareMathOperator{\RG}{RG}
\DeclareMathOperator{\RF}{RF}

\DeclareMathOperator{\Aut}{Aut}
\DeclareMathOperator{\Out}{Out}
\DeclareMathOperator{\Inn}{Inn}

\DeclareMathOperator{\Ball}{B}

\DeclareMathOperator{\F}{F}

\DeclareMathOperator{\Max}{max}

\title{Survey on effective separability}
\author{Jonas Der\'{e}, Michal Ferov, Mark Pengitore}
\date{January 2022}

\begin{document}

\maketitle
\begin{abstract}
    Separability for groups refers to the question which subsets of a group can be detected in its finite quotients. Classically, separability is studied in terms of which classes have a certain separability property, and this question is related to algorithmic problems in groups such as the word problem. A more recent perspective tries to study the order of the smallest finite quotient in which one detects the subset under consideration depending on its complexity, measured using the word norm on a finitely generated group. In this survey, we present what is currently known in the field of effective separability and give an overview of the open questions for several classes of groups.
\end{abstract}
\section{Introduction}
Given a finitely generated group, it is natural to ask how much information can be recovered from its finite quotients, for example in the case of residual finite groups one can distinguish individual elements from each other using finite quotients. More concretely, we say a group $G$ is \textbf{residually finite} if for every pair of distinct elements $g, h \in G$ there is a finite group $Q$ and a (surjective) homomorphism $\varphi \colon G \to Q$ such that $\varphi(g) \neq \varphi(h).$

Properties of this type are called \textbf{separability properties}: a proper nonempty subset $X \subset G$ is \textbf{separable} in $G$ if for every element $g \in G \setminus X$ there exists a finite group $Q$ and a (surjective) homomorphism $\varphi \colon G \to Q$ such that $\varphi(g) \notin \varphi(X).$ Clearly, a group is residually finite if and only if singletons are separable, which is in fact equivalent to the identity element being separable. Separability properties are defined by specifying what specific subsets you want to separable: \textbf{conjugacy separable} groups have separable conjugacy classes, \textbf{cyclic subgroup separable} groups have separable cyclic subgroups, \textbf{locally extended residually finite} (LERF) groups have separable finitely generated subgroups. 

In this survey, we will introduce the readers to the study of quantifying these separability properties, give an overview of the current status of the field, and introduce open problems that are of interest to the broader mathematical community.

\subsection{Motivation}
One of the original motivations for studying separability properties is that they provide an algebraic analogue to decision problems in finitely generated groups. Namely, let $X \subset G$ be a separable subset that is moreover recursively enumerable and that one can effectively construct the image of $X$ under any surjective homomorphism to a finite group. Under these conditions, one can decide if a word in a finite set of generators of $G$ represents an element in $X$ simply by checking finite quotients of $G$. 

Indeed, it was proved by Mal'tsev \cite{Malcev_morphis_finite_groups} and independently in more general setting by McKinsey \cite{mckinsey} that the \textbf{word problem} is decidable for residually finite, finitely presented groups $G = \left< S \mid R \right>$ by running two algorithms in parallel. Given a word $w$ in the generating set $S$, the first algorithm enumerates all finite products of conjugates of relations in $R$ and their inverses and checks whether $w$ appears in this list, whereas the second algorithm enumerates all finite quotients of $G$ and checks whether the image of the element represented by $w$ is nontrivial. In other words, the first algorithm checks for a witness of the triviality of $w$ in $G$ whereas the second algorithm checks for a witness of the nontriviality of $w$ in $G$. Using a similar approach, Mostowski \cite{mostowski} showed the conjugacy problem is solvable for finitely presented conjugacy separable groups. In an analogous manner, LERF groups have a solvable generalised word problem, meaning the membership problem is uniformly solvable for every finitely generated subgroup. Algorithms of this type are known as of \textbf{Mal'tsev-Mostowski} type or \textbf{McKinsey's} algorithms.

Every since Mal'tsev introduced residual finiteness and conjugacy separability to study the word problem and the conjugacy problem, the study of separability properties has been active area of research and has had substantial applications to algebra, dynamics, geometry, and topology. For instance, Scott \cite{Scott1, Scott2} connected separability properties with the topological problem of lifting an immersed submanifold to an embedded submanifold in a finite cover which was instrumental in Agol's resolution of the virtual $\beta_1$ conjecture, the virtual Haken conjecture, and the virtual fibered conjecture \cite{agol}.

Most of the existing work has focused on verifying that different classes of groups satisfy various separability properties. For instance, free groups, finitely generated nilpotent groups, polycyclic groups, closed surface groups, and fundamental groups of geometric $3$-manifolds have all been shown to be residually finite and conjugacy separable \cite{Blackburn,formanek,M_hall,Malcev_morphis_finite_groups,Remeslennikov,Scott1,Scott2,stebe,Wise}. 
It is immediate that every conjugacy separable group is residually finite, since the conjugacy class of the identity element contains only the identity element. The implication in the opposition direction does not hold, with perhaps the easiest example given by Stebe \cite{stebe} and independently by Remeslenikov \cite{Remeslennikov} by proving that $\text{SL}_3(\mathbb{Z})$ is not conjugacy separable. 

In light of the previous discussion, the natural question arises how one can use residual properties such as residual finiteness and conjugacy separability to study finitely generated groups? One approach is by defining a function on the natural numbers that \textbf{measures the complexity} of establishing the residual property by taking the worst case over all words of length at most $n$. While these complexity functions require the selection of a finite generating subset, the asymptotic growth rate as the parameter $n$ goes to infinity is well defined. For instance, for a residually finite, finitely generated group $G$,  Bou-Rabee \cite{Bou_rabee_10} introduced the function $\RF_G \colon \mathbb{N} \to \mathbb{N}$ which quantifies the residual finiteness of $G$. Indeed, if $w$ represents a nontrivial element of $G$ of length at most $n$ with respect to some fixed finite generating subset, then there exists a surjective homomorphism $\varphi \colon G \to Q$ to a finite group such that $\varphi(w) \neq 1$ and where $|Q| \leq \RF_G(n)$. Similarly, Lawton, Louder, and McReynolds \cite{LLM} introduced the function $\Conj_G \colon \mathbb{N} \to \mathbb{N}$ to quantify conjugacy separability. In particular, given two nonconjugate elements $w_1, w_2 \in G$ of length at most $n$, there exists a surjective homomorphism $\varphi \colon G \to Q$ to a finite group such that $\varphi(w_1)$ is not conjugate $\varphi(w_2)$ and where $|Q| \leq \Conj_G(n).$ Similarly, one can quantify separability of a finitely generated subgroups, cyclic subgroups, LERF, and other separability properties as well.

As we mentioned before, separability properties provide an algebraic analogue to decision problems in finitely presented groups. The separability depth function can be then understood as a measure of complexity of the corresponding algorithm of Mal'tsev-Mostowski type. In particular, it allows to dispense of the algorithm which is looking for a positive witness, which we demonstrate for the word problem. Suppose that we are given a finitely presented group $G$, a word $w$ in the generators of $G$, and we know that $\RF_G(n) = f(n),$ where $f \colon \mathbb{N} \to \mathbb{N}$ is some nondecreasing function. We can then enumerate all finite quotients of $G$ of size up to $f(|w|)$. It then follows that if the image of $w$ is trivial in all such quotients, then $w$ must in fact represent the trivial element in $G$.

The study of the quantification of separability properties has great applications to many questions in topology and arithmetic. For instance,  one of the pieces necessary to estimate the index of a Haken cover of a closed hyperbolic $3$-manifold $M$ in terms of the geometric data of $M$ is the quantification of the separability of quasi-convex subgroups of $\pi_1(M)$.  In the direction of arithmetic, quantifying how difficult it is to separate nonidentity elements from the identity using congruence subgroups in the group of integral points of a linear algebraic group $\textbf{G}$ with the congruence subgroup property allows us to quantify strong approximation of $\textbf{G}$. In particular, this quantification allows us to understand how well the group of integral points of $\textbf{G}$ approximates the group of adeles of $\textbf{G}$ with respect to a finite set of places. Since these separability properties have important implications in other fields, our focus will be to compute their asymptotic behaviour for different classes of finitely generated groups. 

\section{Background}
This section contains background for effective separability and a discussion of the framework for this area of group theory. Readers familiar with the definitions may skip this section in the first reading.

\subsection{Separability and the depth function}

Let $G$ be a finitely generated group, and let $X \subset G$ be a proper non-empty subset.  For $g \in G \setminus X$, we define
$$
\D_G(X,g) = \min \left\{|G/N| \hspace{1mm} \mid \: N \trianglelefteq_{f.i.} G \text{ and } g \notin X N \right\}
$$ 
with the understanding that $\D_G(X,g) = \infty$ when no such normal finite index subgroup exists. We call $\D_{G}(X,g)$ the \textbf{depth function of $G$ relative to $X$} and the set $X$ is \textbf{separable} if $\D_G(X,g) < \infty$ for all $g \in G \setminus X.$ We say that $G$ is \textbf{residually finite} if $\{1\}$ is separable, is \textbf{conjugacy separable} if each conjugacy class is separable, and is \textbf{subgroup separable} if each finitely generated subgroup is separable.

The previous definition is an adaption to more general separable subsets of the definition introduced by Bou-Rabee in \cite{Bou_rabee_10} which gave the depth function of $G$ relative to $\{1\}$. To be more specific, we have $\F_{G}(g) = \D_{G}(\{1\}, g)$. It also provides a generalisation of the function $\CD(g,h)$ for nonconjugate elements $g, h \in G$ introduced Lawton, Louder, and McReynolds in \cite{LLM}. In particular, we have that $\CD([g], [h]) = \D_G([g],h) = \D_G([h],g)$ where $[g]$ and $[h]$ indicate the conjugacy classes of $g$ and $h$, respectively, in $G$.

\subsection{Effective separability}
Given a finitely generated group $G$ with finite generating subset $S$, one can define the word length function $\|\cdot\|_S \colon G \to \mathbb{N} \cup \{0\}$ as
\begin{displaymath}
    \|g\|_S = \min \{ |w| \mid w \in F(S) \mbox{ and } w =_G g\}.    
\end{displaymath}
Word length is a standard tool in geometric group theory used to equip $G$ with a left invariant metric $d_{S} \colon G \times G \to \mathbb{N}$ given by $d_{S}(g_1, g_2) = \| g_1^{-1} g_2\|_S$. We will use $\Ball_{G,S}(n)$ to denote the ball of radius $n$ centred around the identity, i.e. $\Ball_{G,S}(n) = \{g \in B \mid \|g\|_S \leq n\}$. When the finite generating subset is clear from context, we will instead write $B_G(n).$

Let $G$ be a finitely generated group equipped with a finite generating subset $S$, and let $X \subset G$ be any separable subset. The \textbf{$X$-separable depth function of $G$} $\RF_{G,X,S} \colon \mathbb{N} \to \mathbb{N}$ is then defined as
$$
\RF_{G,X,S}(n) = \Max\{ \D_{G}(X,g) \: | \: g \in B_G(n) \text{ and } g \notin X \}.
$$
When $X = \{1\}$, we will instead write $\RF_{G,S}(n) = \RF_{G,\{1\},S}(n)$ and call it the \text{residual finiteness depth function}.

While the function $\RF_{G,X,S}(n)$ depends on the finite generating subset, its asymptotic growth does not in the following way. Let $f, g \colon \mathbb{N} \to \mathbb{N}$ be two increasing functions. We say that $f \preceq g$ if there exists a constant $C>0$ such that $f(n) \leq C g(Cn)$ for all natural numbers, and we write $f \simeq g$ if and only if $f \preceq g$ and $g \preceq f$. In particular, we have the following lemma.

\begin{lemma}
\label{lem:generatingset}
Let $G$ be a finitely generated group, and suppose that $X \subset G$ is a separable subset. If $S_1$ and $S_2$ are two finite generating subsets for $G$, then $\RF_{G,X,S_1}(n) \simeq \RF_{G,X, S_2}(n).$
\end{lemma}
In particular, the residual finiteness depth function does not depend on the choice of generating set. Thus, whenever we consider the function $\RF_{G,X}(n)$, we will suppress the reference to the generating subset $S$, taking into account that we only study its asymptotic behaviour.

For $g, h \in G$, we write $g\sim_G h$ if there exists an element $x \in G$ such that $xgx^{-1} = y$ and say they are conjugate. This form an equivalence relation on the elements of the group for which we denote the associated equivalence class of $g$ as $$[g] = \left\{ hgh^{-1} \mid h \in G \right\}.$$ For a conjugacy separable finitely generated group $G$ with finite generating subset we can define the conjugacy separability depth functions $$\Conj_{G,S}(n) = \max\left\{ D_{G,S}([g],h) \mid g, h \in B_G(n), h \notin [g] \right\}.$$ So this function both takes the size of the separating element $h$ and the size of a representative of the conjugacy class $[g]$ into account. Just as in Lemma \ref{lem:generatingset}, one can show that the asymptotic growth does not depend on the generating set $S$ for $G$.

Note that the conjugacy class of the identity element $1 \in G$ is equal to $[1] = \{1 \}$. So for any element $g \neq 1$ of $G$, we have that $D_G(g,[1]) = D_G(g,\left\{1\right\})$ and in particular, the residual finiteness depth function gives a lower bound for the conjugacy separability depth function, i.e.~$\RG_G(n) \leq \Conj_G(n)$. 

\section{Effective residual finiteness}
This section gives an overview of known results and questions related to the residual finiteness depth function $\RF_G(n)$ for different classes of finitely generated groups $G$.

\subsection{Free groups} 
\label{sec:freeRF}One of the most interesting groups to explore residual finiteness is the nonabelian free group of finite rank. Note that we will deal with free abelian groups in Section \ref{sec:RFnilpotent} where we also discuss nilpotent groups. Ever since Bou-Rabee \cite{Bou_rabee_10} initiated the study of effective residual finiteness, there have been many authors who have studied $\RF_{F_k}(n)$ where $F_k$ is the free group of rank $k > 1$. 

The first result in this direction was by Bou-Rabee who used an embedding of $F_k$ into $\SL(2, \mathbb{Z})$ and the Prime Number Theorem to provide polynomial upper bounds for residual finiteness. 
\begin{theorem}
Let $F_k$ be the free group of rank $k > 1$, then $\RF_{F_k}(n) \preceq n^3$.
\end{theorem}

Other than this initial result about the upper bound for $\RF_{F_k}(n)$, no improvements have been made. 
Therefore, we have the natural question which we believe will be an interest to the community.
\begin{question}
Let $F_k$ be the free group of rank $k>1$. Improve the upper bounds for $\RF_{F_k}(n)$
\end{question}

One important thing to note about Bou-Rabee's upper bound for $\RF_{F_k}(n)$ is the representation and number theoretic techniques that were employed. Therefore, it would be interesting to see whether there exists a geometric or topological method for providing a polynomial upper bound for $\RF_{F_k}(n)$. One possibility is to adapt the graph theoretic proof which uses Stallings graph folding to construct a finite index subgroup $H \leq F_k$ which is not necessarily normal that does not contain a given nontrivial element of length at most $n$. In particular, we have the following question.
\begin{question}
Let $F_k$ be the free of rank $k>1.$ Use Stallings graph folding to demonstrate that $\RF_{F_k}(n) \preceq n^3.$
\end{question}

A natural question is to construct asymptotic lower bounds for $\RF_{F_k}(n)$. The first result in this direction is implicitly found in \cite{Bou_rabee_10} where using the fact that $\mathbb{Z} \leq F_k$, it can be shown that $\log(n) \preceq \RF_{F_k}(n).$ However, the first nontrivial bound was provided by Bou-Rabee and McReynolds in \cite{bou_rabee_mcreynolds_2018}.
\begin{theorem}
$n^{1/3} \preceq \RF_{F_k}(n)$.
\end{theorem}
The main idea is the generalisation of least common multiples of integers to general elements in nonabelian free groups.

The next improvement can be found by Kassabov and Matucci in \cite{kassabov_matucci} where the authors are able to improve the lower bound to $n^{2/3}.$
\begin{theorem}
 $n^{2/3} \preceq \RF_{F_k,S}(n).$
\end{theorem}
In this article, the main tool is a result of Lucchini about finite permutation groups which states that if $\Gamma$ is a transitive permutation group of degree $n > 1$ whose point-stabiliser subgroup is cyclic, then $|\Gamma| \leq n^2 - n.$ A natural corollary is that if $\Gamma$ is a finite group with an element $x$ whose order is greater or equal to $\sqrt{|\Gamma|}$, then there is some integer $\ell$ with $\ell < \sqrt{|\Gamma|}$ such that $\left< x^\ell \right>$ is normal in $\Gamma$. This gives a restriction on the upper bound of the order elements in finite simple groups, finally leading to elements that give the lower bound mentioned before.

All of these techniques are trying to construct short laws for all finite groups of size at most $n$. In particular, these laws encode the image of a specified element under all possible homomorphisms from the free group of rank $2$ to finite groups of order at most $n$. Therefore, we have the following definition.
\begin{definition}
Fix an ordered basis $\{x,y\}$ for the free group $F_2$, and let $w \in F_2 \setminus \{1\}$. For any group $G$ define the evaluation map $G \times G \to G$ (also denoted $w$) by $w(g,h) = \pi_{(g,h)}(w)$ where $\pi_{(g,h)}$ is the (unique) homomorphism $F_2 \to G$ extending $x \to g$, $y \to h$. We call $w$ a \textbf{law} for $G$ if $w(G \times G) = \{1\}$.
\end{definition}
In other words, a \emph{law} for a group $G$ is an equation which holds identically in $G$. The interest in laws is a classical subject, growing out of the work of Birkhoff \cite{birkhoff} in universal algebra, and further developed by many authors (see \cite{neumann} and its references). Moreover, certain specific laws have been the subject of intense study over the years, particularly power laws. 

We could also define word maps $G^k \to G$ associated to elements for $F_k$ for any $k \geq 2$ which would allow us to define laws for $G$ within $F_k$. However, it turns out that not much is lost by restricting to special case $k=2$. Indeed, we have that if $k>2$, the standard embeddings of $F_k$ into $F_2$ associate to every law $w \in F_k$ for $G$ a law $\tilde{w} \in F_2$ for $G$ with a length that depends linearly on the length of $w$. Conversely, we have an inclusion of a basis for $F_2$ into a basis for $F_k$ which turns every law for $G$ in $F_2$ into a law for $F_k$ of the same length.  We also have that a nontrivial element $w \in F_1 \cong \mathbb{Z}$ is a law for $G$ if and only if the exponent of $G$ divides $w$ when viewed as an integer.


The application to residual finiteness for free groups can be found in the following lemma.

\begin{lemma}
Let $f \colon \mathbb{N}\cup\{\infty\} \to \mathbb{N}\cup\{\infty\}$ be a strictly increasing function with inverse given by $f^{-1}$. Suppose that there are nontrivial elements $w_n \in F_2$ of length at most $f(n)$ which is a law for all finite groups of order at most $n$, we then have that $f^{-1}(n) \preceq \RF_{F_2}(n)$.
\end{lemma}
 
Using this lemma, Thom \cite{Thom_laws} improved the asymptotic lower bounds for residual finiteness for nonabelian free groups as seen in the following theorem.
\begin{theorem}
For all $n \in \mathbb{N}$, there exists a word $w_n \in F_2$ of length
$$
\text{O} \left( \frac{n \log \log (n)^{9/2}}{\log(n)^2} \right)
$$
such that for every finite group $G$ such that $|G| \leq n$, $w_n$ is a law for $G$ 
\end{theorem}
Thom approached constructing these laws by reducing it to two disjoint cases. The first is for short laws of finite solvable groups of order at most n which has been investigated by many different authors. The second involves investigating laws for finite simple groups which uses the classification of finite simple groups and thus considerably heavier machinery.

The study of short laws for finite simple groups was continued by Bradford and Thom in \cite{Thom_bradford} as seen in the following theorem.
\begin{theorem}
For all $n \in \mathbb{N}$ there exists a word $w_n \in F_2$ of length at most
$$
\text{O}(n^{3/2} \log(n)^3 \log^*(n)^2)
$$
such that for every finite group $G$ satisfying $|G| \leq n$, $w_n$ is a law for $G$.
\end{theorem}
In the statement of the previous theorem, $\log^*(n)$ denotes the iterated logarithm, i.e. the smallest natural number $k$ such that the $k$-fold application of $\log$ to $n$ yields a real number less than $1$. Note that $log^*(n)$ grows slower than any iteration of logarithms.

As a direct application of the above theorem, we have the currently best known lower bound. 
\begin{theorem}
$\frac{n^{3/2}}{\log(n)^{9/2 + \epsilon}} \preceq \RF_{F_2}(n)$
\end{theorem}
It is likely that the conclusion of the above theorem is best possible up to logarithmic factors. One reason to see this is that for simple groups $G$ where $|G| \leq n$ that are not isomorphic to $\text{PSL}_2(q),$ $\text{PSL}_3(q),$ or $\text{PSU}_2(q)$ where $q$ is a prime power have a law of length at most $n^{2/3}$ by \cite[Proposition 3.2]{Thom_bradford}. Hence, to strengthen the lower bound, one needs to understand laws for finite simple groups given by $\text{PSL}_2(q),$ $\text{PSL}_3(q),$ and $\text{PSU}_2(q)$ where $q$ is a prime power. The main difficulty to moving forward is noting that the these groups have elements of large order relative to their size.

\begin{question}
Let $F_k$ be the free group of rank $k>1$ with a finite generating subset $S$. Compute the precise asymptotic behaviour of $\RF_{F_k}(n)$. 
\end{question}

Note that \cite[Remark 9]{kassabov_matucci} claims that the methods of \cite{hadad} can be used to show that the short law satisfied by all groups of the form $\SL_2(R)$ where $R$ is a finite commutative ring of size at most $N$ has length $CN^2$. If this is true, then one can improve the upper bound from $\RF_{F_K}(n) \preceq n^3$ to $\RF_{F_K}(n) \preceq n^{3/2}$. Combined with the known lower bounds, this leads to the conjecture that $\RF_{F_k}(n) \simeq n^{\frac{3}{2}}$ for any $k > 1$.

\subsection{Arithmetic groups}
For arithmetic groups, the following question motivates the interest in the study of the asymptotic behaviour of $\RF_{G}$ when $G$ is an arithmetic group, which we will define below. It was originally asked by D. Mostow during a geometry seminar at Yale University in December 2009.
\begin{question}
Does asymptotic information of residual finiteness characterise arithmetic subgroups of a given linear algebraic group?
\end{question}
The first major step towards addressing this question is due to Bou-Rabee and Kaletha \cite{bou_rabee_kaletha}. Here the authors demonstrate that for a fixed Chevalley group, all $S$-arithmetic subgroups $G$ of an arithmetic group $\textbf{G}$ share the same asymptotic growth for $\RF_G$. To be more precise, a Chevalley group $\textbf{G}$ for this theorem is a split simple algebraic group that is not necessarily connected. We use the term $S$-arithmetic subgroup of $\textbf{G}$ to denote any subgroup $G$ of $\textbf{G}(\mathbb{C})$ which is commensurable with $\textbf{G}(\mathcal{S}_{K,f})$, where $K \subset \mathbb{C}$ is a subfield with $\mathcal{O}_K$ its ring of integers and $f \in \mathcal{O}_K \setminus \{0\}$. That is, a $S$-arithmetic subgroup of $\textbf{G}$ in the usual definition for a number field $K$ and some finite set of places of $K$ which contain the Archimedean ones. However, we will allow $K$ and $S$ to vary.

\begin{theorem}
Let $\textbf{G}$ be a Chevalley group of rank at least $2$, $K$ a number field, and $f \in \mathcal{O}_K \setminus \{0\}.$ If $\Gamma$ is a finitely generated subgroup of $\textbf{G}(\mathbb{C})$ with the property that $\Gamma \cap \textbf{G}(\mathcal{O}_{K,f})$ is of finite index in $\textbf{G}(\mathcal{O}_{K,f})$, then $\RF_{\Gamma}(n) \approx n^{\dim(\textbf{G})}.$
\end{theorem}

The main ideas used for their proof is the structure of split semi-simple group schemes, results on the congruence subgroup problem, Moy-Prasad filtrations, Selberg's Lemma, the Prime Number Theorem, and the Chebotar\"{e}v density theorem. They also use in essential way results found in Lubotzky-Mozes-Raghunathan \cite{lmr}. 

This result was later extended in the context of arithmetic groups defined over purely transcendental extensions of a finite field by Franz \cite{franz}.
\begin{theorem}
Let $\textbf{G}$ be a Chevalley group of rank at least $2$. Let $K$ be a purely transcendental extension of a finite field, and let $\Gamma < \textbf{G}(K)$. Let $\mathcal{O} = \mathbb{F}_p[t]$. If $\Gamma \cap \textbf{G}(\mathcal{O}) \leq \textbf{G}(\mathcal{O})$ has finite index, then $\RF_{\Gamma}(n) \approx n^{\dim(\textbf{G})}.$
\end{theorem}
In this paper, the author uses an effective form of the Chebotar\"{e}v density theorem. For the lower bounds, he uses the congruence subgroup property and properties of the Chevalley groups and the associated graded Lie algebras.

For Chevalley groups of rank $1$, upper bounds have also been found. The following theorem by Franz is the best upper bound which has be provided for general affine algebraic group schemes  defined over $\mathbb{Z}$. The upper bound for fields of characteristic $0$ was provided by Bou-Rabee and Kaletha \cite{bou_rabee_kaletha} and was later extended to fields of characteristic $p$ by Franz \cite{franz}. While we mention this theorem in the context of Chevalley groups of arbitrary rank, the theorem applies to broader class of all affine algebraic group schemes defined over $\mathbb{Z}.$
\begin{theorem}
Let $\textbf{G}$ be an affine algebraic group scheme defined over $\mathbb{Z}$. Let $K$ be a field, and let $\Gamma \leq \textbf{G}(K)$ be finitely generated. If the characteristic of $K$ is $0$ or $K$ is a purely transcendental extension of a finite field, then $\RF_{\Gamma}(n) \preceq n^{\dim(\textbf{G})}.$
\end{theorem}

What remains to be understood in the study of asymptotic lower bounds for residual finiteness of arithmetic groups is in the context of Chevalley groups of rank $1$. One can provide provide lower bounds for these groups using the fact that they contain nonabelian free groups, it would be interesting to see if a lower bounds for residual finiteness for these groups using information of the ambient Chevalley group which leads to the following question.
\begin{question}
Let $\textbf{G}$ be a Chevalley group of rank $1$, and let $K$ be a field. Suppose that $K$ is either a number field or a purely transcendental extension of a finite field. Let $\mathcal{O}$ be either the associated ring of integers if $K$ is a number field or let $\mathcal{O}$ be $\mathbb{F}_p[t]$ if $K$ has characteristic $p$. Let $\Gamma \leq \textbf{K}$ be a finitely generated group that is commensurable with $\textbf{G}(\mathcal{O})$. Find an asymptotic lower bound for $\RF_{\Gamma}(n)$ in terms of the geometry of $\textbf{G}$ and properties of the field $K$.
\end{question}
\subsection{Virtually nilpotent and solvable groups}
\label{sec:RFnilpotent}
For finitely generated abelian groups, the residual finiteness depth function is completely understood by \cite[Corollary 2.3.]{Bou_rabee_10}.

\begin{theorem}
Let $A$ be an infinite, finitely generated abelian group, then $$\RF_A(n) \simeq \log(n).$$
\end{theorem}

The result follows from the asymptotic behaviour of the function $$\psi(r) = \lcm(1, \ldots, r),$$ which is well-known due to its relation to the Prime Number Theorem. Indeed, it holds that $\displaystyle \lim_{r \to \infty} \frac{ \log(\psi(r))}{r} = 1$, leading to the aforementioned result by some estimates.

For all finitely generated groups $G$ containing a finite index residually finite subgroup $H \le G$, we know that $$\RF_G(n) \preceq \RF_H(n)^{[G:H]}$$ by a standard argument, leading to an upper bound of the form $\log(n)^k$ for virtually abelian groups. The exact function remains unknown though for the class of virtually abelian groups, and no example of a finitely generated virtually abelian group $G$ is known for which the residual finiteness depth function is not logarithmic. 
\begin{question}
Find a finitely generated virtually abelian group $G$ for which $$\RF_G(n) \not \simeq \log(n).$$
\end{question}

For nilpotent groups, the situation is more complicated. Since any finite nilpotent group is a direct sum of $p$-groups, a quotient of minimal order for a given element will always be a $p$-group. The minimal size of such a quotient may depend on properties of the prime number $p$, as demonstrated by the example in Pengitore \cite{pengitore2}. The sharpest results for the residual finiteness depth function are given in Pengitore \cite{pengitore3}, where $T(N)$ denotes the normal subgroup of torsion elements in a finitely generated nilpotent group $N$.
\begin{theorem}
Let $N$ be a finitely generated nilpotent group such that $N/T(N)$ has step length $c>1.$ There exist (explicit) natural numbers $k_1$ and $k_2$ where $k_1 \geq c+1$ and $k_2$ is less than the Hirsch length of $N$ such that $$\log(n)^{k_1} \preceq \RF_N(n) \preceq \log(n)^{k_2}.$$ In general, the numbers $k_1$ and $k_2$ are not equal. 
\end{theorem}

Although the exact function $\RF_N(n)$ is unknown for a general nilpotent group $N$, it is possible to compute it in several concrete examples, e.g.~in the Heisenberg groups. All the known examples satisfy the property that the residual finiteness depth function is a quasi-isometric invariant, leading to the following question.

\begin{question}
Do there exist quasi-isometric nilpotent groups $N_1 \sim_{QI} N_2$ such that $\RF_{N_1}(n) \not \simeq \RF_{N_2}(n)$?
\end{question}

The exact behaviour of the residual finiteness depth function for virtually nilpotent groups is also widely open. 

Next we explore what is known for finitely generated residually finite solvable groups. 
For solvable groups $G$ that are linear, such as polycyclic groups, we always have $\RF_G(n) \preceq n^k$ for some $k > 0$, since this upper bound holds for all linear groups by Bou-Rabee and McReynolds in \cite[Theorem 1.1.]{bm15}. Moreover, the group $G$ is virtually nilpotent if and only if the residual finiteness depth function is bounded above by $\log(n)^k$ for some $k > 0$, see \cite[Theorem 1.1.]{bm11}, giving a new lower bound for solvable groups that are not virtually nilpotent. These are the only general bounds known for polycyclic groups, and there are no direct upper bounds that do not use linearity.

\begin{question}
Give new upper bounds for $\RF_G(n)$ for virtually polycyclic groups $G$ that do not use linearity which only depend on the Hirsch length. 
\end{question}

In \cite{pengitore4}, Pengitore gave a lower bound for certain solvable groups $G$, based on the existence of distorted elements in the Fitting subgroup, which is the subgroup generated by normal nilpotent subgroups of $G$. The main idea is that for elements in the Fitting subgroup, we can apply the lower bound for nilpotent groups, where elements of length $n$ can be longer in a subgroup because of subgroup distortion. In particular, it implies that for any such group with an element $x$ in the Fitting subgroup that is exponentially distorted, $\RF_G(n)$ must be at least linear. Depending on which step of the lower central series of the nilpotent normal subgroup of Fitting subgroup contains the element $x$ this lower bound can be further improved. Moreover, this paper gives the first non-trivial lower bounds for $\RF_G(n)$ for solvable groups of infinite Pr\"{u}fer rank that satisfy the conditions of its theorem.

Since these lower bounds only take into account certain elements in the Fitting subgroup, it is natural to study how to improve these by using more information of the group.

\begin{question}
Give lower bounds for virtually polycyclic groups that take into account the full structure of the group.
\end{question}

Another type of lower bound for solvable groups $G$ of finite Pr\"ufer rank is given by Bou-Rabee and McReynolds in \cite{bm11}. Here, the authors apply the relation between the residual finiteness depth functions $\RF_G(n)$, the word growth of the group $G$ and the normal subgroup growth, which measures the number of normal subgroups of a given index. For a solvable group of finite Pr\"ufer rank, we have that the normal subgroup growth is bounded above by $n^m$ and this leads to a lower bound $n^{\frac{2}{m}} \preceq \RG_G(n)$ when these groups have exponential word growth (or equivalently, the groups are not virtually nilpotent). 

For residual finiteness of linear finitely generated solvable group $G$ with no exponentially distorted elements or that are of infinite Pr\"{u}fer rank, we have that there exists some infinite field $K$ such that $G \leq \GL(m,K)$ where $K$ may have nonzero characteristic and will not be a finite algebraic extension of $\mathbb{Q}$. Since $G$ is finitely generated, the field generated by the coefficients of the matrices over $\mathbb{Q}$ or $\mathbb{F}_p$, depending on the characteristic of $K$, is finitely generated. Thus, it is a finite extension of a transcendental extension of $\mathbb{Q}$ or $\mathbb{F}_p$ of finite transcendence degree. Using restriction of scalars (or corestriction) we may assume this extension is purely transcendental. Hence, we have that $G$ is subgroup of $\GL(m, F(x_1, \ldots, x_s))$ where $F = \mathbb{Q}$ or $F = \mathbb{F}_p$. In particular, the coefficients of elements of $G$ would lie in the ring $R[\frac{1}{S}]$ where $R = \mathbb{Z}[x_1, \ldots, x_s]$ or $R = \mathbb{F}_p[x_1, \ldots, x_s]$ and $S$ a finite set of elements of $R$. 

Using basic algebra, we can then bound the degree and size of coefficients of the matrix representative of a word in $G$ in terms of its word length with respect to a fixed finite generating subset. As before, we can take the Zariski closure of $G$ in $\GL(m, F(x_1, \ldots, x_s))$ to obtain a linear algebraic group $\textbf{G}$ where $G \leq \GL(m, R[\frac{1}{S}]) \cap \textbf{G}$. 
However, the main difficulty compared to before is that we have polynomial matrix entries now. When $\textbf{G}$ is defined over $\mathbb{Z}[x_1, \ldots, x_s]/S$, we may reduce $\textbf{G}$ by taking the mod $p$ reduction to obtain $\textbf{G}_p$ which is then defined over $\mathbb{F}_p[x_1, \ldots, x_m]/S'$ where $S'$ is a finite collection of polynomials. We then further reduce modulo some irreducible ideal $\mathcal{I}$ in $\mathbb{F}_p[x_1, \ldots, x_s]$ where $\mathcal{I} \cap S' = \emptyset$ to obtain a linear algebraic group $\textbf{G}_{p^\ell}$ defined over $\mathbb{F}_{p^\ell}$ where $\ell$ is the index of $\mathcal{I}$ in $\mathbb{F}_p[x_1, \ldots, x_s]$. At this point, we need to find a matrix whose polynomial entries are highly divisible by ideals in $\mathbb{F}_p[x_1, \ldots, x_s]/S$ which presents difficulties of its own. The story when $\textbf{G}$ is defined over $\mathbb{F}_p$ is similar. Therefore, we have the following question.
\begin{question}
Let $G \leq \GL(n, K)$ be a finitely generated group where $K$ is an infinite field. Let $\textbf{G}$ be the Zariski closure of $G$ in $\GL(n,K)$. Find upper and lower bounds for $\RF_G(n)$ in terms of the geometric and algebraic properties of $\textbf{G}.$
\end{question}

One final estimate we want to mention is that asymptotic upper and lower bounds have been found for the lamplighter groups $\RF_{\mathbb{F}_p \wr \mathbb{Z}}(n)$ where $p$ is prime by Bou-Rabee, Chen, and Timashova \cite{lamplighter_rf} as seen in the following theorem.
\begin{theorem}
Let $\Gamma = \mathbb{F}_p \wr \mathbb{Z}$ where $p$ is prime. Then
$$
n^{3/2} \preceq \RF_{\Gamma}(n) \preceq n^2.
$$
\end{theorem}
The main technique the authors use in this article is that they show the classical representation of lamplighter groups have a congruence subgroup property that works well with the coarse geometry of the group. There are several questions that one may ask in light of this theorem. The first is as follows.
\begin{question}
Find precise asymptotic bounds for residual finiteness of the lamplighter groups.
\end{question}

Another question is to see if the above methods can be generalised for other linear wreath products such as $\mathbb{Z} \wr \mathbb{Z}$ or the higher rank of the lampligther groups given by $\mathbb{F}_p \wr \mathbb{Z}^k$ for $k>1.$ Thus, we have the following question.
\begin{question}
Let $\Gamma$ be either $\mathbb{Z} \wr \mathbb{Z}$ or $\mathbb{F}_p \wr \mathbb{Z}^k$ for $k>1.$ Provide nontrivial asymptotic bounds for residual finiteness of $\Gamma.$
\end{question}

We finish this subsection by exploring what is known for residual finiteness of a general solvable group. It is important to note that not all finitely presented solvable groups are residual finite. Kharlampovich provided infinitely many examples of finitely presented solvable groups of derived length $3$ which do not have solvable word problem. In particular, these groups are not residual finite. Thus, we have a nontrivial restriction when we assume that our solvable groups are residually finite. One may hope that for the class of residually finite finitely generated solvable groups $G$ that there exists a function $f$ that provides a universal upper bound for residual finiteness of this groups. However, as the following theorem of Khalarmpovich, Myasnikov, and Sapir demonstrates, this is not possible. Moreover, this theorem demonstrates that residual finiteness provides an extremely inefficient solution to the word problem for these whereas there exists a polynomial time solution for the word problem.
\begin{theorem}\label{horrible_rf}
For every recursive function $f$, there is a residually finite finitely presented solvable group $G$ of derived length $3$ where $f(n) \preceq \RF_G(n)$. Moreover, one can assume that the word problem in $G$ can be solved in at most polynomial time.
\end{theorem}

This theorem motivates the following question. For each class of functions, such as polynomial, intermediate growth, exponential, super exponential, etc., does there exist a characterisation of the collection of finite presented solvable groups whose asymptotic growth of residual finiteness lies within that class.
\begin{question}
Characterise the collection of residually finite finitely presented solvable groups whose residual finiteness growth is asymptotically polynomial, intermediate, exponential, etc.
\end{question}

\subsection{Mapping class groups}
Though we know that $\RF_{\Mod(\Sigma_g)}(n)$ is a nowhere infinite function by virtue of $\Mod(\Sigma_g)$ being residual finite, no explicit upper bounds for residual finiteness of $\Mod(\Sigma_g)$ have been written down other than for the genus 2. Indeed, by Bigelow and Budney \cite{genus2_linear} we have that $\Mod(\Sigma_2)$ is linear which implies that there are polynomial upper bounds for $\RF_{\Mod(\Sigma_2)}(n)$ by \cite{bm15}. For genus $g > 2$, having a polynomial upper bound for $\RF_{\Mod(\Sigma_g)}(n)$ would give a strong indication that $\Mod(\Sigma_g)$ is linear. 

Using results from the literature, one can provide rather coarse bounds for the residual finiteness of $\text{Mod}_g$ when $g > 2$ in the following way. Let $f \in \Mod(\Sigma_g)$ be an element of length at most $n$ with respect to some fixed generating subset. Via the action of $\Mod(\Sigma_g)$ on the homology of $\Sigma_g$, which preserves the algebraic intersection number, we have the symplectic representation $\Phi \colon \Mod(\Sigma_g) \to \text{Sp}(2g, \mathbb{Z})$ which is well known to be surjective with kernel given by the Torelli group $\mathcal{T}(\Sigma_g)$. Thus, if $f$ acts nontrivially on the homology of $\Sigma_g$, we have that $\Phi(f) \neq 1$. Therefore, there exists an integer $d$ such that there exists a surjective homomorphism to a finite group $\psi \colon \text{Sp}(2g, \mathbb{Z}) \to Q$ such that $\psi(\Phi(f)) \neq 1$ and where $|Q| \leq C (\|g\|)^d$ for some universal constant $C>0$. In order to find an upper bound, we may hence assume that $f \in \mathcal{T}(\Sigma_g)$. Since the group $\mathcal{T}(\Sigma_g)$ is torsion free, the element $f$ has infinite order.

Fix a finite generating set $T$ for $\pi_1(\Sigma_g)$. We have that $f$ acts on the set of oriented isotopy classes of simply closed curves in $\Sigma_g$. An observation in \cite{koberda_mangahas} states that that there exists a universal constant $\lambda > 1$ such that if $w \in \pi_1(\Sigma_g)$, then the shortest representative for $f([w])$ in $\pi_1(\Sigma_g)$ has length at most $\lambda^n \|w\|_T$ where $n$ is the length of the mapping class element $f$. Moreover, basic hyperbolic geometry implies that if $c_1, c_2$ are a pair of filling curves in minimal position, then the subset of the mapping class group that preserves the unordered pair of isotopy classes $\{[c_1], [c_2]\}$ is a finite group. Since $f$ has infinite order, there exists an $i \in \{1, 2\}$ such that $f([c_i]) \neq [c_i]$. Letting $C = \text{max}\{\|w_1\|_T, \|w_2\|_T\}$ where $c_i$ is in the based homotopy class of $w_i$, we have that the shortest representative of either $f([w_1])$ or $f([w_2])$ has length at most $C \lambda^n$ where $n$ is the word length of $f$ as an element of the mapping class group. All together this implies that for any infinite order element $f \in \Mod(\Sigma_g)$ of length at most $n$, we have that there exists a word $w \in \pi_1(\Sigma_g)$ such that $[w]$ and $f([w])$ are distinct conjugacy classes with representatives of length at most $C \lambda^n$.

Thus, in order to distinguish $f$ from the identity in a finite quotient, we use what are known as principal congruence quotients of $\Mod(\Sigma_g)$ which are given in the following way. For a finite index characteristic subgroup $K \leq \pi_1(\Sigma_g)$, the Dehn-Nielsen-Baer theorem implies there exists an induced map $\varphi_K \colon \Mod(\Sigma_g) \to \Out(\pi_1(\Sigma_g) / K)$. If we find a finite index characteristic subgroup $K$ of $\pi_1(\Sigma_g)$ such that the conjugacy classes $[w]$ and $f([w])$ remain distinct, then $f$ has an non-trivial action on the conjugacy classes of $\pi_1(\Sigma_g)/K$ and hence $\varphi_K(f) \neq 1$ in $\Out(\pi(\Sigma_g)/K)$. Thus, to provide an upper bound on the sizes of the image of $\varphi_k$, there are 3 quantities that need to be estimated. 

The first is to provide a quantification of conjugacy separability for $\pi_1(\Sigma_g)$. That is, if we are given two non-conjugate elements $g,h \in \pi_1(\Sigma_g)$ of word length at most $n$, we want to find a finite quotient $\psi \colon \pi_1(\Sigma_g) \to Q$ where $\psi(g)$ and $\psi(h)$ remain non-conjugate and where $|Q|$ is bounded in terms of $n$. As we will explain in the next section, this is bounded by $n^{n^2}$.

The second quantity that needs to be estimated is the index of the characteristic core of a finite index normal subgroup. In particular, we are interested when given a normal subgroup $N \leq \pi_1(\Sigma_g)$ of index at most $n$, bounding the index of the maximal characteristic subgroup of $\pi_1(S)$ contained in $N$ in terms of $[\pi_1(\Sigma_g) : N]$. By \cite[Corollary 2.8]{subgroup_growth} and the fact that $\pi_1(\Sigma_g)$ surjects onto a nonabelian free group, we have that the number of distinct normal subgroups of index $n$ is asymptotic to $n^{\log n}$. Since the image under any automorphism of a normal subgroup is normal, we have that the index characteristic core of a normal subgroup of index at most $n$ is a constant multiple of $n^{n^{\log(n)}}$. 

Finally, we need to provide an estimate for order of $\Im(\varphi_K)$. All we have at this point is a subgroup of a quotient of $\Aut(\pi_1(\Sigma_g)/K)$ which is a subgroup of the symmetric group on $|\pi_1(\Sigma_g)/K|$ letters. Therefore, we have that $|\Im(\varphi_k)| \leq |\pi_1(S)/K)|!$. 

Combining this all together, we have that for a mapping class $f$ of word length $n$, there exists a surjection onto a finite group $\varphi \colon \Mod(\Sigma_g) \to Q$ such that $\varphi(f) \neq 1$ and where $|Q|$ is bounded above by $(m^{m^2})^{((m^{m^2})^{\lceil\log(m^{m^2}) \rceil})}!$ with $m = C \lambda^n$ up to some fixed constant. As this is a huge upper bound, the natural question is whether we can improve it. 
\begin{question}
Give improved upper bounds for the residual finiteness of the mapping class group of the closed orientable surface of genus greater than $2$.
\end{question}

In order to demonstrate that $f(n) \preceq \RF_{\Mod(\Sigma_g)}(n)$ for any increasing function $f(n)$, we need to demonstrate that there exists an infinite sequence of mapping classes $\{g_i\}$ such that the minimal finite quotient of $\Mod(\Sigma_g)$ in which $g_i$ does not vanish has order bounded below by $f(\|g_i\|)$. Finding such a sequence elements that give a nontrivial lower bound for a general residually finite group is hard in general due to the plethora of finite quotients of $\Mod(\Sigma_g)$. Indeed, a general finite quotient of $\Mod(\Sigma_g)$ can be quite mysterious, so it is necessary to restrict to subclass of finite quotients which are more explicit. By using the class of principal congruence quotients, we find an associated complexity function, for which we can study residual finiteness for $\Mod(\Sigma_g)$.

\begin{question}
Provide asymptotic lower bounds for residual finiteness of the mapping class group of closed orientable surface of genus greater than $2$ with respect to the principal congruence quotients.
\end{question}

The Dehn-Nielsen-Baer theorem implies that $\Mod(\Sigma_g)$ comes with a class of finite index subgroups known as congruence subgroups. For each chacteristic subgroup of $K$ of $\pi_1(\Sigma_g)$, we have an induced homomorphism $\varphi$ from $\Mod(\Sigma_g)$ to $\text{Out}(\pi_1(\Sigma_g) / K)$ whose kernel is known as a \emph{principal congruence subgroup}, and any finite index subgroup that contains a principal congruence subgroup is referred to as a \emph{congruence subgroup}. Whether each finite index subgroup of $\Mod(\Sigma_g)$ is a congruence subgroup is known as the \emph{congruence subgroup problem} (see \cite{ivanov}), and while the congruence subgroup property is not the main focus of this subsection, a positive answer would allow us a means of understanding the finite index subgroup structure of $\Mod(\Sigma_g)$ which would greatly benefit the study of residual finiteness of $\Mod(\Sigma_g)$.

\subsection{RAAGs}
Given that all right angled Artin groups $G$ (RAAGs) are linear, we have that $\RF_G(n) \preceq n^d$ for some integer $d >0$. However, it would be interesting to demonstrate a polynomial upper bound for residual finiteness of RAAGs using the inherent geometry of the group, for instance through the use of the Salvetti complex. We also note that nonabelian RAAGs contain a nonabelian free group which gives superlinear lower bounds for residual finiteness as described in Section \ref{sec:freeRF}. In the same spirit as the upper bound, it would be interesting to find nontrivial lower bounds for residual finiteness of a RAAG starting from the geometry of the group. This leads to the following question.
\begin{question}
Let $G$ be a nonabelian RAAG. Find asymptotic upper and lower bounds for $\RF_G(n)$ using the geometry of $G$.
\end{question}

\subsection{Closed geometric $3$-manifolds}
Let $M$ be a closed orientable hyperbolic $3$-manifold, and let $G = \pi_1(M)$ be its fundamental group. We know by the resolution of the virtual Haken conjecture by Agol \cite{agol}, building on the work of Wise \cite{Wise}, that all $G$ is always linear, and in particular, we have that $G$ is virtually a subgroup of a RAAG. Thus, we would be able to appeal to linearity of $G$ or to the fact that $G$ is a subgroup of a RAAG to provide polynomial asymptotic upper bounds for $\RF_G(n).$ However, it would be interesting to see if there are asymptotic bounds for $\RF_G(n)$ that reflect the geometry of the manifold $M$. Hence, we have the following question.
\begin{question}
Let $M$ be a closed orientable hyperbolic $3$-manifold, and let $G = \pi_1(M)$. Provide asymptotic upper and lower bounds for $\RF_G(n)$ using the geometry of $M$.
\end{question}

\subsection{$\text{Out}(F_k)$ and Outer automorphism groups of RAAGs}

We know by Charney and Vogtmann \cite{charney_vogtman} and independently Minasyan \cite{cs_raag_minasyan} that $\Out(\Gamma)$ is residually finite when $\Gamma$ is a RAAGs. However, this collection includes groups as disparate as $\GL(n, \mathbb{Z})$ and $\Out(F_k)$ where $F_k$ is the free group of rank $k \geq 3$, where the latter is not linear as a consequence of Formanek and Procesis \cite{formanek_procesi}. In some cases of RAAGs, such as when $\Gamma = \mathbb{Z}^n$ or $\Gamma = F_2$, the asymptotic behaviour of residual finiteness of $\Out(\Gamma)$ is completely understood as a consequence of their linearity. However, for examples such as $\Out(F_k)$ where $k>2$, the asymptotic behaviour of residual finiteness is poorly understood at best. One can provide superlinear lower bounds for $\RF_{\Out(F_k)}(n)$ by virtue of these groups containing nonabelian free groups, and, in a similar way as for the mapping class group, one can provide really coarse upper bounds for $\RF_{\Out(F_k)}(n)$. However, we expect that there are asymptotic upper and lower bounds for $\Out(F_k)$ that come from the geometry of the free group and natural objects on which $\Out(F_k)$ acts, such as outer space. Therefore, we have the following natural question which we feel should be of interest to the broader group theory community.
\begin{question}
\label{question:outer}
Let $F_k$ be the free group of rank $k \geq 3$. Provide nontrivial upper and lower asymptotic bounds for $\RF_{\Out(F_k)}(n)$.
\end{question}

\subsection{Branch groups}
The last class of which we discuss results about the asymptotic residual finiteness are groups which act by automorphisms on regular rooted trees, which are known to be nonlinear. This includes the so-called branch groups, namely groups that admit a lattice of subnormal subgroups with the branching structure following the structure of the tree on which the groups act. 

This class contains many example of groups with remarkable algebraic properties, the first of which is known as the \emph{first Grigorchuck group} $\Gamma$. It is an infinite finitely generated group whose elements have order equal to some power of $2$, every proper quotient is finite, the group is commensurable with $\Gamma \times \Gamma$ and has intermediate growth. The asymptotic behaviour of residual finiteness for the first Grigorchuk group is known by Bou-Rabee \cite{Bou_rabee_10}.
\begin{theorem}
Let $\Gamma$ be the first Grigorchuk group, then $\RF_{\Gamma}(n) \approx 2^n.$
\end{theorem}

In order to discuss results on regular branch groups, we will consider certain subgroups of the automorphism group of a regular rooted tree. Let $X$ be a finite alphabet with $|X| \geq 2.$ The vertex set of the tree $T_X$ is the set of finite sequences over $X$. Two sequences are connected by an edge when one can be obtained from the other by right-adjunction of a letter in $X$. The root is the empty sequence $\emptyset$, and the children of $v$ are the vertices $vx$ for $x \in X$. The set $X^n \subset T_X$ is called the \emph{nth level of the tree $T_X$}. An \emph{automorphism} of the tree is a bijective graph homomorphism of $T_X.$

Let $g \in \Aut(T_X)$ be an automorphism of the rooted tree $T_X$. Consider a vertex $v \in T_X$ and the subtrees
$$
vT_X = \{vw \: | \: w\ \in T_X\} \quad \text{and} \quad g(v)T_X = \{g(v) w \: | \: w \in T_X\}.
$$
The map $v T_X \to g(v) T_X$ is a morphism of rooted trees with $vT_X$ and $g(v)T_X$ both naturally isomorphic to $T_X.$  By identifying $vT_X$ and $g(v)T_X$ with $T_X$, we obtain an automorphism $g|_v \colon T_X \to T_X$, which is uniquely defined by the condition
$$
g(vw) = g(v) g|_v(w)
$$
for all $w \in T_X.$ We call the automorphism $g|_v$ the \emph{restriction of $g$ on $v$}. A subgroup $G \leq \Aut(T_X)$ is \emph{self-similar} if for every $g \in G$
 and every $v \in T_X$, we have $g|_v \in G.$

We also define the notion of contraction of an action of a self-similar group on a regular rooted tree.

\begin{definition}
Let $T_X$ be a regular rooted tree on the alphabet $X$ where $|X| \geq 2$, and let $G \leq \Aut(T_X)$ be a self-similar group with a finite generating subset $S$. The number
$$
\lambda_{G,T_X} = \limsup_{n \to \infty} \sqrt[n]{\limsup_{\|g\|_S \to \infty} \Max_{v \in X^n} \frac{\|g|_v\|_S}{\|g\|_S}}
$$
is called the \emph{contraction coefficient}. A self-similar group $G \leq \Aut(T_X)$ is called \emph{contracting} if $\lambda_{G,T_X} < 1$.
 \end{definition}

For $G \leq \Aut(T_X)$ and $v \in T_X$, the \emph{vertex stabiliser} is the subgroup consisting of the automorphisms that fix $v$:
$$
\text{Stab}_G(v) = \{g \in G \: | \: g(v) = v\}.
$$
The \emph{nth level stabiliser} (also known as principal congruence subgroup) is given by
$$
\text{Stab}_{G}(n) = \bigcap_{v \in X^n}\text{Stab}_G(v).
$$
The \emph{rigid stabiliser} $\text{rist}_G(v)$ of a vertex $v \in T_X$ is the subgroup of $G$ of all automorphisms acting non-trivially only on the vertices of the form $vu$ with $u \in T_X:$
$$
\text{rist}_G(v) = \{g \in G \: | \: g(w) = w \text{ for all } w \notin vT_X\}.
$$
The \emph{nth level rigid stabiliser} is given by
$$
\text{rist}_G(n) = \left< \text{rist}_G(v) \: | \: v \in X^n \right> \cong \prod_{v \in X^n} \text{rist}_G(v)
$$
which is the subgroup generated by the union of the rigid stabilisers of the vertices of the nth level.

Any $g \in \text{Stab}_G(n)$ can be identified in a natural way with the sequence $$(g_1, \ldots, g_{|X|^n})$$ of elements in $\Aut(T_X)$ where $g_i = g|_v$ is the restriction of $g$ to the vertex $v$ of level $n$ having the number $i$ in some ordering of the vertices in the nth level $(1 \leq i \leq |X|^n)$. We say that $g$ is of \emph{level $n$} if $g \in \text{Stab}_G(n) \setminus \text{Stab}_G(n+1)$ and write $g = (g_1, \ldots, g_{|X|^n})_n.$ We say that a subgroup $K$ \emph{geometrically contains $K^{|X|^n}$} for some $n \geq 1$, if for every $k_1, \ldots, k_{|X|^n} \in K$, there is some element $k \in K$ such that $k = (k_1, \ldots, k_{|X|^i})_i$. We call a level transitive group $G \leq \Aut(T_X)$ \emph{branch} if $\text{rist}_G(n)$ is finite index in $G$ for all $n \geq 1.$ 

For the sake of the results that we will be citing we will restrict ourselves to the following important type of branch groups.
\begin{definition}
A level transitive group $G \leq \Aut(T_X)$ is \emph{regular-branch} if there exists a finite index subgroup such that $K$ geometrically contains $K^{|X|}$ of finite index.
\end{definition}

The study of asymptotic residual finiteness for this more general class of groups that act on regular rooted trees is seen in the following theorem of Bou-Rabee and Myropolska \cite{bou_rabee_myropolska} who provide at least exponential upper bounds for the asymptotic behaviour of residual finiteness.
\begin{theorem}
Let $G$ be a finitely generated group acting on a rooted d-regular tree. Suppose that $G$ is regular-branch and contracting with contraction coefficient $\lambda < 1$. Then
$$
\RF_{G}(n) \preceq 2^{n^{\frac{1}{\log_d(1/ \lambda)}}}.
$$
\end{theorem}
In \cite[Lemma 2.13.9]{self_similar_groups}, it is shown that the contracting coefficient $\lambda$ satisfies $\frac{1}{\log_d(1 / \lambda)} \geq 1$. Hence, the upper bound achieved the above theorem is at least exponential.

For regular-branch groups that are not necessarily contracting the following super-polynomial lower bound is provided for their residual finiteness growth due to Bou-Rabee and Myropolska \cite{bou_rabee_myropolska}.
\begin{theorem}
Let $G$ be a finitely generated regular-branch group acting on a rooted $d$-regular tree. Then
$$
2^{n^{\frac{1}{\log_d(\lambda)}}} \preceq \RF_G(n).
$$
\end{theorem}

This next theorem by Bou-Rabee and Myropolska \cite{bou_rabee_myropolska} provides super exponential lower bounds for the Gupta-Sidki p-group when $p \geq 5$ .
\begin{theorem}
Let $r>0$. Then there exists a prime $p$ such that if $G_p$ is the Gupta-Sidki $p$-group, then
$$
2^{n^{\frac{\ln(p)}{\ln(3)}}} \preceq \RF_{G_p}(n).
$$
In particular, $G_p$ for $p \geq 5$ has super-exponential residual finiteness growth.
\end{theorem}

The main idea behind both theorems for the lower bound is to use the congruence subgroup property with respect to level stabilisers to find elements which are difficult to distinguish from the identity using finite quotients. In particular, these elements are of level $n$ which implies that if you have a finite quotient which factors through a quotient by the level $k$ stabiliser, then these elements must be trivial in these quotients. One can then give bounds for the size of the quotient by the $n$-level stabiliser and give estimates for the size of minimal length elements of level $n$. For the upper bounds, one estimates how deep the level $n$-stabiliser needs to not contain your given element in terms of the word length.

There are many things one can study for residual finiteness growth of branch groups. For instance, it would be interesting if one could precisely compute the residual finiteness growth of Gupta-Sidki groups for all primes $p$. Therefore, we have the following modest question.
\begin{question}
Let $G_p$ be the Gupta-Sidki $p$-group where $p$ is prime. Compute the precise asymptotic growth of $\RF_{G_p}(n).$
\end{question}
\section{Conjugacy separability} 

Unlike the residual finiteness depth function, the conjugacy separability depth function does not behave well with respect to taking finite degree extensions or finite index subgroups. Indeed, Goryaga \cite{goryaga} gave an example of a non-conjugacy separable group that contained a conjugacy separable subgroup of index two, Martino and Minasyan \cite{martino_minasyan} constructed infinitely many examples of conjugacy separable groups that contained a subgroup of finite index that was not conjugacy separable. Therefore, results of this type are usually nontrivial to find.

\subsection{Virtually nilpotent and solvable groups}

For a finitely generated abelian group, every conjugacy class is equal to a singleton; hence the conjugacy separability depth function is equivalent to the residual finiteness depth function.

\begin{theorem}
Let $A$ be a finitely generated abelian group, then $$\Conj_A(n) \simeq \RF_A(n) \simeq \log(n).$$
\end{theorem}

The situation for finitely generated nilpotent groups $N$ is very different, although conjugacy classes are still well-behaved, in the sense that if two elements $g, h \in N$ are conjugate, then $g = h x$ for some $x \in \gamma_2(N) = [N,N]$. This allows for induction on nilpotence class, since by definition the lower central series $\gamma_i(N)$ defined inductively as $\gamma_{i+1}(N) = [N,\gamma_i(N)]$ becomes trivial after a finite number of steps. A first step in this direction was made by Pengitore \cite{Pengitore_1} by giving lower and upper bounds for $\Conj_N(n)$ for a finitely generated nilpotent group. Later, this was generalised to virtually nilpotent groups by Der\'{e} and Pengitore in the paper \cite{DerePengitore_1}, where the authors use the notion of twisted conjugacy separability. A twisted conjugacy class $[g]_\varphi$ for an automorphism $\varphi: G \to G$ is defined as the elements $$[g]_\varphi = \left\{ h g \varphi(h)^{-1} \mid h \in G\right\}.$$ For the identity automorphism, we indeed retrieve the regular conjugacy classes. As a consequence of an effective version for separating twisted conjugacy classes, we get the following. 
\begin{theorem}
Let $G$ be a finitely generated virtually nilpotent group. If $G$ is virtually abelian, then there exists $k > 0$ such that $$ \log(n) \preceq \Conj_G(n) \preceq \log(n)^k, $$ if not, then there exists $1 < k_1 < k_2$ such that $$n^{k_1} \preceq \Conj_G(n) \preceq n^{k_2}.$$
\end{theorem}
The first statement is not given explicitly in the paper, but follows from the methods in the paper. For almost all nilpotent groups, the exact function is unknown, with the exception of the Heisenberg groups which were studied in Pengitore \cite{Pengitore_1}. Again, it is an open question whether the conjugacy separability depth function is a quasi-isometric invariant for (virtually) nilpotent groups.

\begin{question}
Do there exist finitely generated nilpotent groups $N_1$ and $N_2$ such that $N_1 \sim_{QI} N_2$ but $\Conj_{N_1} \not \simeq \Conj_{N_2}$?
\end{question}

For polycyclic groups $G$ that are not virtually nilpotent, the behaviour of $\Conj_G(n)$ is widely open, although it is expected to be exponential.

\begin{question}
Show that if $G$ is virtually polycyclic but not virtually nilpotent, then $\Conj_G(n) \simeq \exp(n)$.
\end{question}

Beyond polycyclic groups, there are some results on effective conjugacy separability for other classes of finitely generated conjugacy separable solvable groups. For instance, the following result by Ferov and Pengitore \cite{ferov_pengitore} provides exact asymptotic bounds for conjugacy separability depth function for lamplighter groups, which provides the first asymptotic characterisation of conjugacy separability for a class of groups outside of abelian groups. Furthermore, exponential lower bounds and super exponential upper bounds have been provided for conjugacy separability for wreath products of abelian groups where the base group is finite and the acting is infinite.
\begin{theorem}
Let $A$ be a finite abelian group, and let $B$ be an infinite, finitely generated abelian group. If $B$ has torsion free rank $1$, then 
$$
\Conj_{A \wr B}(n) \approx 2^n.
$$ More generally, if $B$ is an infinite finitely generated abelian group of torsion free rank $k>1$, then
$$
2^n \preceq \Conj_{A \wr B}(n) \preceq 2^{n^{2k}}.
$$
\end{theorem}

Continuing in this direction, super exponential upper and lower bounds have been provide for conjugacy separability of wreath products of abelian groups where both the base group and acting group are infinite by Ferov and Pengitore in \cite{ferov_pengitore2}.
\begin{theorem}
Let $A$ and $B$ be infinite finitely generated abelian group. If $B$ has torsion free rank $1$, then
$$
(\log n)^n \preceq \Conj_{A \wr B}(n) \preceq (\log n)^{n^2}.
$$
If $B$ has torsion free rank $k>1,$ then
$$
(\log n)^n \preceq \Conj_{A \wr B}(n) \preceq (\log n)^{n^{2k+2}}.
$$
\end{theorem}

As an application of the Magnus embedding of the free metabelian group into the wreath product of infinite, finitely generated abelian groups, we obtain the following upper asymptotic bound for conjugacy separability of free metabelian group of rank $k$.
\begin{theorem}
If $S_{2,m}$ is the free metabelian group of rank $k$, then
$$
\Conj_{S_{2,m}}(n) \preceq (\log n)^{n^{2m+2}}.
$$
\end{theorem}
A natural consequence following the above theorem is to see if one can provide nontrivial lower bounds for conjugacy separability of the free metabelian group. More generally, it is well known that all free solvable groups of arbitrary rank and derived length are conjugacy separable, so one may be interested in the asymptotic behaviour of conjugacy separability of these groups.
\begin{question}
Let $S_{d,m}$ be the free solvable group of derived length $d$ and rank $m$. Compute the asymptotic behaviour of $\Conj_{S_{d,m}}(n).$
\end{question}
Combining results from \cite{ferov_pengitore}, \cite{ferov_pengitore2} and \cite{ferov_pengitore3}, we obtain the following asymptotic bounds for wreath products of nilpotent groups where the base group is abelian and the acting group is a general nilpotent group. Note that a wreath product of nilpotent groups if conjugacy separable if and only if the base group is abelian.
\begin{theorem}
Let $A$ be a finitely generated abelian group, and let $N$ be an infinite, finitely generated abelian group. There exists a natural number $d$ such that if $A$ is finite, then
$$
2^n \preceq \Conj_{A \wr N}(n) \preceq 2^{n^{n^{d}}}.
$$
If $A$ is infinite, then
$$
(\log n)^n \preceq \Conj_{A \wr N}(n) \preceq n^{n^{n^{d}}}.
$$
\end{theorem}

\subsection{Improving upper and lower bounds for free groups and surface groups}
Another class of groups for which the conjugacy separability depth function has been studied are free groups and surface groups in Lawton, Louder, and McReynolds \cite{LLM}. Both types of groups satisfy the property that for every $g \in G$, there exists a representation $$\rho_g: G \to \SL(n_g,\C)$$ such that if $h \in G$ is not conjugate to $g$, then $\Tr(\rho_g(g)) \neq \Tr(\rho_g(h))$. The coefficients of the elements $\rho_g(G) \subset \SL(n_g,\C)$ lie in a certain ring $R$, and then the authors use the fact that the difference in trace is non-zero to find a finite quotient field $F$ of the ring $R$ such that the trace is still non-zero, leading to a finite group $\SL(n_g,F)$ in which the elements are not conjugate. By estimating the dimension $n_g$ in terms of the word length of $g$, the following upper bound is found.

\begin{theorem}
\label{thm:conjfree}
Let $G$ be a surface group or a free group of finite rank, then the conjugacy separability depth function satisfies $$\Conj_G(n) \preceq n^{n^2}.$$
\end{theorem}

It is unclear how far this bound is from being sharp. The only known lower bound is given by the residual finiteness depth function $\RF_G$, so in particular for the free group the best known lower bound follows from Section \ref{sec:freeRF}.

\begin{question}
Give new lower bounds for $\Conj_G(n)$ where $G$ is a surface group or a free group of finite rank. 
\end{question}

\subsection{Giving upper and lower bounds for RAAGs}
We know by Minasyan \cite{cs_raag_minasyan} that RAAGs are hereditarily conjugacy separable, that is all finite index subgroups of RAAG are conjugacy separable. Other than the nonabelian free groups and free abelian groups, no work has been done on the study of the asymptotic behaviour on the conjugacy separability of a RAAG. Therefore, one may be interested in the study of asymptotic conjugacy separability for these groups. 
\begin{question}
Let $G$ be a finite index subgroup of general right angled Artin group. Provide asymptotic upper and lower bounds for $\Conj_{G}(n).$
\end{question}

Moreover, it would be interesting to relate the asymptotic behaviour of conjugacy separability to the the geometric properties defining graph of the RAAG.
\begin{question}
Let $\Gamma$ be a connected simplicial graph, and let $G(\Gamma)$ be the associated right angled Artin group. Relate the asymptotic behaviour of $\Conj_{G(\Gamma)}(n)$ to the geometric properties of $\Gamma.$
\end{question}

\subsection{Fundamental groups of closed orientable $3$-manifolds}
It was demonstrated by Hamilton, Wilton, and Zalesskii in \cite{compact} that fundamental groups of closed orientable $3$-manifolds are conjugacy separable. However, there are no effective bounds so far, so it would be interesting to study the asymptotic behaviour of the conjugacy separability depth function for this class of groups and relate it to the pieces in the JSJ decomposition of the manifold.
\begin{question}
Let $M$ be a closed orientable $3$-manifold. Find asymptotic bounds for $\Conj_{\pi_1(M)}(n)$. 
\end{question}

One interesting avenue of study is when the closed $3$-manifolds admits a specific geometry. One may be interested also in relating conjugacy separability to factors in the JSJ decomposition of a general closed orientable $3$-manifold.
\begin{question}
Let $M$ be a closed orientable $3$-manifold. If $M$ admits one of the $8$ Thurston geometries, can one relate the asymptotic behaviour of $\Conj_{\pi_1(M)}(n)$ to the geometry it admits? More generally, how does the asymptotic behaviour of $\Conj_{\pi_1(M)}(n)$ relate to the asymptotic behaviour of conjugacy separability of the pieces that show up the in the JSJ decomposition of $M$?
\end{question}

\subsection{Grigorchuk's group and Gupta-Sidki $p$-groups}
It is known that Grigorchuk's group and the Gupta-Sidki $p$-groups for $p \neq 2$ are conjugacy separable by Gupta and Sidki\cite{gupta_sidki}. While the residual finiteness depth function has been for studied and almost completely characterised for Grigorchuk's group and more generally for a large class of branch groups, the upper bounds and lower bounds for asymptotic behaviour of conjugacy separability is unknown. We have a straight forward lower bound for conjugacy separability for these group coming from the lower bounds for asymptotic residual finiteness. Other than that, little is known. Hence, we have the following question.
\begin{question}
Let $G$ be Grigorchuck's' group or one of the Gupta-Sidki $p$-groups for $p \neq 2$. Provide asymptotic upper bounds for $\Conj_G(n)$. Furthermore, provide asymptotic lower bounds not coming from the asymptotic behaviour of residual finiteness.
\end{question}

\subsection{Computability of the Conjugacy Problem and Conjugacy Separability}
As mentioned by Theorem \ref{horrible_rf}, the complexity of residual finiteness can be arbitrarily high whereas the complexity of the word problem is at most polynomial. Therefore, it would be interesting to see if there is a similar result that can be found for conjugacy separability as seen in the following question.
\begin{question}
Let $f$ be an nondecreasing recursive function $f$. Does there exist a finitely presented conjugacy separable group $G$ where $f(n) \preceq \Conj_G(n)$ but where the conjugacy problem can be solved in polynomial time?
\end{question}

It is well known that there exist finitely presented groups which have  simple solutions to the word problem and arbitrarily complex solutions to the conjugacy problem. One may ask then if there exists finitely presented conjugacy separable groups where the asymptotic behaviour of residual finiteness is much simpler than the asymptotic behaviour of conjugacy separability.
\begin{question}
Let $f$ be a nondecreasing recursive function. Does there exist a finitely presented conjugacy separable group $G$ where $\RF_G(n) \preceq n^d$ for some natural number $d$ and where $f(n) \preceq \Conj_G(n)?$
\end{question}

We finish this subsection by asking what kinds of functions can arise as the asymptotic behaviour of conjugacy behaviour such as intermediate growth. Hence, we have this modest last question.
\begin{question}
Does there exist a finitely presented conjugacy separable group where such that $n^d \preceq \Conj_G(n)$ for all natural numbers $d$ and where $\Conj_G(n)$ has asymptotic complexity strictly less than that of $2^n?$
\end{question}

\subsection{Applications to residual finiteness of outer automorphism groups}
In this subsection  we combine known results to give a first known upper bound on residual finiteness depth function for $\Out(F_k)$ where $k>2$. First, we prove a following general lemma.
\begin{lemma}
\label{lemma:Out_upper}
Let $G$ be a finitely generated conjugacy separable group and suppose that $\Aut(G)$ is also finitely generated.
Suppose that there are functions $a,b,c,d,e \colon \mathbb{N} \to \mathbb{N}$ such that
\begin{itemize}
    \item[(i)] $|\Aut(Q)| \preceq a(n)$ whenever $Q$ is a finite quotient of $G$ of size at most $n$, ;
    \item[(ii)] $\mathop{nsub}_G(n) \leq b(n)$, where $\mathop{nsub}_G(n)$ is the number of normal subgroups of $G$ of index $n$;
    \item[(ii)] $\Conj_G(n) \preceq c(n)$;
    \item[(iv)] $\|\alpha(w)\| \preceq d(\|w\|)$ for all $w \in G$, whenever $\alpha \in \Ball_{\Aut(G)}(n)$;
    \item[(v)] for every $\alpha \in \Ball_{\Aut(G)}(n)\setminus\Inn(G)$ there is $w \in G$ such that $\alpha(w) \not\sim_G w$ and $\|w\| \leq e(n)$.
\end{itemize}
 Then
   \begin{displaymath}
       \RF_{\Out(G)}(n) \preceq a\left( c(d(e(n)))^{b(c(d(e(n))))}\right).
   \end{displaymath}
\end{lemma}
\begin{proof}
    Suppose that $\alpha \in \Ball_{\Aut(G)}(n)\setminus\Inn(G)$ is given, then we see that there is an element $w \in \Ball_{G}(e(n))$ such that $w$ and $\alpha(w)$ are not conjugate. We see that $\alpha(w) \in \Ball_G(d(e(n)))$. Therefore, there is $N \unlhd G$ such that $wN$ and $\alpha(w)N$ are not conjugate in $G/N$ and $|G/N| \leq c(d(e(n)))$. Let $K$ be the characteristic core of $N$, i.e. 
        $$K = \bigcap_{\gamma \in \Aut(G)} \gamma(N).$$
    As $|G/\gamma(N)| = G/N$ for every $\gamma \in \Aut(G)$, we see that
    $$|G/K| \leq |G/N|^{b(G/N)} \leq c(d(e(n)))^{b(c(d(e(n))))}.$$
    As $K$ is characteristic, the canonical projection $\pi \colon G \to G/K$ induces a homomorphism $\tilde{\pi} \colon \Aut(G) \to \Aut(G/K)$. From the construction of $K$ we see that $\tilde{\pi}(\alpha)(wK)$ is not conjugate to $wK$, therefore $\tilde{\pi}(\alpha) \notin \Inn(G/K)$. I follows that
    \begin{displaymath}
        |\Out(G/K)| \leq |\Aut(G/K)| \leq a\left(|G/K|\right) \leq a\left( c(d(e(n)))^{b(c(d(e(n))))}\right),
    \end{displaymath}
    which concludes our proof.
\end{proof}

We now apply the lemma to the setting of free groups. First, we note that:
\begin{itemize}
    \item[(i)] it can be easily seen that if $Q$ is a finite quotient of $F_k$, then $|\Aut(Q)| \leq |Q|^k$ as every generator must be mapped to one of the elements of $Q$;
    \item[(ii)] it was showed by Newman \cite{newman_counting} that $\mathop{nsub}_{F_k}(n) \preceq n (n!)^{r-1}$;
    \item[(iii)] by Theorem \ref{thm:conjfree} we have that $\Conj_{F_k} \preceq n^{n^2}$;
    \item[(iv)] as $\Aut(F_k)$ is generated by the automorphisms corresponding to the elementary Nielsen transformations, we see that for every $\alpha \in \Ball_{\Aut(F_k)}(n)$ we have that $\|\alpha(w)\| \leq 2^{\|w\|}$;
    \item[(v)] one can show that for every $\alpha \in \Aut(F_k) \setminus \Inn(F_k)$ there exists and element $w \in F_k$ with $\|w\| \leq 2$ such that $w$ and $\alpha(w)$ are not conjugate in $F_k$.
\end{itemize}
Then we immediately get the following upper bound.
\begin{theorem}
\label{theorem:out_Fk_upper}
    \begin{displaymath}
            \RF_{\Out(F_k)}(n) \preceq
            \left(
                \left({2^n}^{2^{2n}}\right)^{{2^n}^{2^{2n}}\left({2^n}^{2^{2n}}!\right)^{k-1}}
            \right)^k.
    \end{displaymath}
\end{theorem}
This immediately provides an answer to Question \ref{question:outer}, albeit a slightly unsatisfying one. Hence we pose a new question.
\begin{question}
    Let $F_k$ be the free group of rank $k \geq 3$. Provide nontrivial upper asymptotic bounds for $\RF_{\Out(F_k)}(n)$ better than those given by Theorem \ref{theorem:out_Fk_upper}.
\end{question}

\section{Subgroup separability}
In this section we list the known results on $\RF_{G,H}(n)$ where $H$ is a finitely generated subgroup of a finitely generated group $G$. 

\subsection{Free groups and surface groups}

Let $G$ be a free group of rank $r > 1$ or a surface group of genus $g>1$ containing a finitely generated subgroup $H$. In \cite{Louder_McReynolds_Patel}, Louder, McReynolds, and Patel construct a representation $\rho_H: G \to \GL(V)$ such that the subgroup $\rho_H(H)$ is closed for the subspace topology on $\rho(G)$ induced by the Zariski topology on $\GL(V)$. As a consequence, they find the following result on effective separability.

\begin{theorem}
Let $G$ be a free group of rank $r > 1$ or a surface group of genus $g>1$ containing a finitely generated subgroup $H$, then there exists $d > 0$ such that $$\RF_{G,H}(n) \preceq n^d.$$
\end{theorem}

The only known lower bound is given by the residual finiteness depth function $\RF_{G}(n)$. Note that Hagen and Patel \cite{Hagen_Patel} study a related notion where they search for every subgroup $H$ and $g \notin H$ a subgroup $H_0$ of finite index with $H \subset H_0$ and $g \notin H_0$. The authors show that in this case, a linear upper bound for the index of $H_0$ can be found, but this does not lead to a polynomial upper bound for the normal core of $H_0$. 

\subsection{Nilpotent groups}

In \cite{dere_pengitore_2}, Der\'{e} and Pengitore not only give a bound on the function $\RF_{G,H}(n)$ for $G$ a torsion-free finitely generated nilpotent group, but they also give an estimate on the constants depending on the size of the subgroup $H$. Indeed, if $S$ is a finite generating set for $G$, one can define a norm on finitely generated subgroups $H$ as $\Vert H \Vert_S \leq n$ if and only if $B_{G,S}(n) \cap H$ generates the subgroup $H$. In this way, one could define a new residual finiteness depth function which also takes into account the norm of the group $H$, namely $$\Sub_{G,S}(n) = \max \left\{D_G(H,g) \mid \Vert H \Vert_S \leq n, \Vert g \Vert_S \leq n \right\}.$$
Again, it is a standard argument to show that $\Sub_{G,S}(n)$ does not depend on the generating set for $S$. 

\begin{theorem}
Let $N$ be a torsion-free finitely generated nilpotent group of rank $r$, then $$\log(n) \preceq \RF_{N,H}(n) \preceq \log(n)^r$$ for every subgroup $H$ of $N$. Moreover, there exists $k > 0$ such that $$n \preceq \Sub_N(n) \preceq n^k.$$
\end{theorem}

For infinite abelian groups $A$, is it known that $\RF_{A,H}(n) \simeq \log(n)$ if $H$ has infinite index in $A$ (otherwise the function $\RF_{A,H}$ is bounded) and that $\Sub_A(n) \simeq n$. For general nilpotent groups, the exact function is unknown.

\begin{question}
If $N$ is a torsion-free finitely generated nilpotent group with infinite index subgroup $H$, does there exists $d > 0$ such that$$\RF_{N,H}(n) \simeq \log(n)^d?$$ Does there exists $k > 0$ such that $$\Sub_{N}(n) \simeq n^k?$$
\end{question}

A natural next step is consider polycyclic groups. These groups are well known to be subgroup separable and that every subgroup is finite generated. Thus, one can ask if there is a similar split in behaviour between virtually nilpotent groups and virtually polycyclic groups that are not virtually nilpotent in the asymptotic behaviour of separability of subgroup. One can also ask if there is a similar characterisation of the asymptotic behaviour of subgroup separability. Therefore, we have the following question.
\begin{question}
Let $G$ be a virtually polycyclic group that is not virtually nilpotent, and let $H \leq G$ be a subgroup. Show that there exist integers $k_1, k_2$ such that either
$$
\log(n)^{k_1} \preceq \RF_{G,H}(n) \preceq \log(n)^{k_2}
$$
or
$$
n^{k_1} \preceq \RF_{G,H}(n) \preceq n^{k_2}.
$$
Moreover, determine the conditions on the subgroup $H$ so that $\RF_{G,H}(n)$ has polynomial growth. Finally, demonstrate that $\Sub_{G}(n)$ has exponential growth. To be more specific, show that
$$
\Sub_{G}(n) \approx 2^n.
$$
\end{question}

\subsection{Grigorchuk's group and the Gupta-Sidki p-groups}
The last interesting collection of groups which have separable subgroups and satisfy subgroup separability that we want to mention is Grigorchuk's group and the Gupta-Sidki $p$-groups. We have by Grigorchuk and Wilson \cite{grigorchuk_subgroup_separability} that Grigorchuk's group is subgroup separable, and it was demonstrated by Garrido \cite{garrido}, Francoeur and Leeman \cite[Corollary 4.24]{subgroup_induction} that Gupta-Sidki $G_p$ for primes $p$ are subgroup separable. Hence, we have the following questions.

\begin{question}
Let $G$ be the Grigorchuk's group or the Gupta-Sidki $p$-groups, and let $H \leq G.$ Find asymptotic upper and lower bounds for $\RF_{G,H}(n)$. Additionally, find asymptotic upper and lower bounds for $\Sub_{G}(n).$
\end{question}

\bibliographystyle{plain}
\bibliography{bibs}
\end{document}